\documentclass{article}
\usepackage[a4paper]{geometry}
\usepackage{graphicx} % Required for inserting images
\usepackage{mathtools,amssymb,bm,amsthm}
\usepackage{tikz-cd}

\newcommand{\field}{\mathbb{K}}
\newcommand{\hecke}[1]{\mathcal{H}_{#1}}
\newcommand{\cible}{\mathcal{V}}
\newcommand{\source}[1]{\Gamma^{#1}_q \mathcal{V}}
\newcommand{\quantumfunctor}[1]{\mathcal{P}^{#1}_q}
\newcommand{\polyfunctor}[1]{\mathcal{P}^{#1}_1}
\newcommand{\symgroup}[1]{\mathfrak{S}_{#1}}

\newcommand{\twist}[1]{#1^{(1)_q}}
\newcommand{\tensorfunctor}[1]{\mbox{$\bigotimes^{#1}$}}

\DeclareMathOperator{\End}{End}
\DeclareMathOperator{\Hom}{Hom}

\DeclareMathOperator{\id}{id}

\DeclareMathOperator{\Ext}{Ext}
\DeclareMathOperator{\Image}{im}
\DeclareMathOperator{\Induction}{Ind}

\newtheorem{theorem}{Theorem}[section]
\newtheorem*{theorem*}{Theorem}
\newtheorem{lemma}[theorem]{Lemma}
\newtheorem{proposition}[theorem]{Proposition}
\newtheorem{corollary}[theorem]{Corollary}
\newtheorem{conjecture}[theorem]{Conjecture}
\theoremstyle{remark}
  \newtheorem{remark}[theorem]{Remark}
  
  \newtheorem*{warning}{Warning}
\theoremstyle{definition}
  \newtheorem{definition}[theorem]{Definition}
  
  \newtheorem{problem}[theorem]{Problem}
  \newtheorem{example}[theorem]{Example}
  \newtheorem*{examples}{Examples}
  \newtheorem{notation}[theorem]{Notation}

\title{Quantum Troesch complexes}
\author{Théo Deturck}
\date{}

\begin{document}

\maketitle

\begin{abstract}
    We study the effect of quantum Frobenius twist on Ext-groups in the category of quantum polynomial, and prove that the existence of type of complexes, called quantum Troesch complexes, enables the construction of a spectral sequence computing the Ext groups of twisted functors from the knowledge of Ext-groups of the original functors. We then construct quantum Troesch complexes in the special case where the parameters of the quantum deformation is a root of unity of order $3$.
\end{abstract}

\section{Introduction}

Let $\field$ be a field, and $q \in \field^{\times}$ be a non-zero scalar. In \cite{friedlander1997cohomology}, Friedlander and Suslin introduced the category of strict polynomial functors, a category particularly useful for studying the representation theory of $\mathrm{GL}_n(\field)$, and of Schur algebras. This category was deformed in \cite{hong2017quantum} to obtain the category of quantum polynomial functors. The latter is related to $q$-Schur algebras $S_q(n;d)$ in the same way that strict polynomial functors are related to Schur algebras. More specifically, a quantum polynomial functor $F$ of degree $d$ can be seen as a family $F(n)$ of $S_q(n;d)$-modules with some compatibilities between them. The evaluation $F \mapsto F(n)$ then defines an equivalence of category between the category of quantum polynomial functors of degree $d$ and the category of finite-dimensional $S_q(n;d)$-modules, provided that $n \geq d$. In particular, it gives isomorphisms
\[
    \Ext^*_{\quantumfunctor{}}(F,G) \xrightarrow{\simeq} \Ext^*_{S_q(n;d)}(F(n),G(n))
\]
between Ext-groups, when $n \geq d$.

When $q$ is a root of unity of odd order, a quantum polynomial functor $\twist{F}$ can be constructed from a strict polynomial functor $F$, using the quantum Frobenius twist. The representations of the $q$-Schur algebra obtained in this way play an important role in the representation theory of quantum Schur algebras. For example, each irreducible representation of the $q$-Schur algebra decomposes as a tensor product of the form $M \otimes N^{(1)_q}$, where $M$ belongs to a finite family of representations of the $q$-Schur algebra, and $N$ is a representation of the classical Schur algebra. Hence, it may be interesting to compare the Ext-groups $\Ext^*_{\quantumfunctor{}}(\twist{F},\twist{G})$ and $\Ext^*_{\polyfunctor{}}(F,G)$. In \cite{théo2025extgroupcategoryquantumpolynomial}, we made the following conjecture.
\begin{conjecture}\label{conj}
    For all strict polynomial functors $F,G$, there is a graded isomorphism
    \[
        \Ext^k_{\quantumfunctor{}}(F^{(1)_q},G^{(1)_q}) \simeq \bigoplus_{i+j=k} \Ext^i_{\polyfunctor{}}(F,G^j_{E}).
    \]
    where $G^j_{E}$ is a strict polynomial functor constructed from $G$ in an explicit way (definitions \ref{def:parametre} and \ref{def:E1}).
\end{conjecture}

We proved in \cite{théo2025extgroupcategoryquantumpolynomial} that this conjecture holds in characteristic $0$ or in characteristic $p$ strictly greater than the degrees of $F$ and $G$. The purpose of this article is to present some advances towards a proof of conjecture \ref{conj} independent of the characteristic. To do this, we will consider quantum versions of Troesch complexes (and adapted to another context in \cite{drupieski2022superized}), introduced by Troesch in \cite{troesch2005resolution}, and used by Touzé in \cite{touze2010universal} and \cite{touze2012troesch} to tackle the analogous problem with the classical Frobenius twist (also solved by Chalupnik in \cite{chalupnik2015derived}).

The paper is divided into two parts. The first part (section 2) introduces the concept of quantum Troesch complex (see definition \ref{def:qTroesch}) and explains what can be done if we have a quantum Troesch complex at our disposal. After briefly reviewing the general theory of quantum polynomial functors and of $\ell$-complexes, we show how one can obtain, \textit{assuming the existence of quantum Troesch complexes}, isomorphisms
\[
    \Ext^*(\twist{F},\twist{(S^\mu)}) \simeq \Hom_{\polyfunctor{}}(F,(S^\mu_{E})^*)
\]
natural in $F$ and $S^\mu$ (the $S^\mu$, where $\mu$ is a tuple of non-negative integers, are strict polynomial functors which form a cogenerating family of injective objects). We then deduce from this the existence of a twisting spectral sequence, as in \cite[Theorem 7.1]{touze2012troesch}.
\begin{theorem}\label{theo:intro}
    Suppose the existence of quantum Troesch complexes. Let $F,G$ be strict polynomial functors. Then there is a first quadrant spectral sequence, natural in $F$ and $G$,
    \[
        E_2^{s,t}(F,G) = \Ext^s_{\polyfunctor{}}(F,G_{E}^t) \implies \Ext_{\quantumfunctor{}}^{s+t}(\twist{F},\twist{G}).
    \]
    Moreover, if $F_1,F_2,G_1,G_2$ are strict polynomial functors, there is a pairing of spectral sequences
    \[
        E^{*,*}(F_1,G_1) \otimes E^{*,*}(F_2,G_2) \to E^{*,*}(F_1 \otimes F_2, G_1 \otimes G_2) 
    \]
    (see \cite[3.9]{benson1991representations} for a definition) which coincides with the cup product on the second page and on the abutment. If this spectral sequence collapses at the second page for some particular $F$ and $G$, then conjecture \ref{conj} holds for these particular choice of $F$ and $G$.
\end{theorem}
For example, if $\Ext^s_{\polyfunctor{}}(F,G_{E}^t) = 0$ whenever $s$ is odd, then the conjecture \ref{conj} holds for the functors $F$ and $G$.

The second part of the paper (section 3) is dedicated to the construction of the quantum Troesch complexes when $q^3 = 1$. This is achevied in theorem \ref{theo:cohomologyBd}. In particular, this shows that theorem \ref{theo:intro} holds in this case. Actually, we prove a bit more than what is required in the definition of quantum Troesch complex, namely we endow the quantum Troesch complexes with a structure of graded differential $\quantumfunctor{}$-algebra, and prove that it is an exponential functor (see definition \ref{def : exponential functor}). This enables us to simplify the computation of the homology of the Troesch complexes. The main difficulty is constructing an appropriate differential and product. Whereas the differential itself is quite natural to consider, the product needed is an ad-hoc deformation of a more natural product. More precisely, the differential is a linear combination of the form $\delta_1 + \lambda \delta_2$, with $\delta_1$ and $\delta_2$ being two degree-$1$ maps which are natural to consider, and $\lambda$ (which is not actually a constant) may not be the most natural choice, but it is not difficult to determine. 
We observe that when $q^\ell=1$, the most natural product on quantum Troesch complexes comes from identifying this complex with the algebra of quantum $\ell \times n$ matrices, and taking the associated product. But contrarily to the ad-hoc product we discovered when $\ell=3$, this most natural product does not seem to be compatible with the differential of definition \ref{def : Troesch pour l=3}, nor with any differentials the author has considered. This leaves the following problem open for the moment.
\begin{problem}\label{problem}
Assume that $q^\ell=1$, for $\ell> 3$. Does there exist a Troesch complex in the sense of definition \ref{def:qTroesch}?
\end{problem}
% Meanwhile, on the quantum Troesch complexes, the most natural product comes from identifying it with the algebra of quantum $\ell \times n$ matrices, and taking this product. But this would not be compatible with the differential, nor with any differentials the author has considered, even changing what "being compatible" means. This is the reason we were not able to generalize our construction even for the case $\ell = 5$.

In this article, we have shown (in theorem \ref{theo:intro}) how to use quantum Troesch complexes as an ingredient in the computation of $\Ext$-group between twisted functors. We note that Troesch complexes have found other uses in the classical setting, for example in the construction of universal cohomology classes leading to the proof of finite cohomological generation \cite{touze2010}. Our theorem \ref{theo:cohomologyBd} and the resolution of problem \ref{problem} could be interesting in the perspective of proving analogues of this in the quantum setting.

\section{General theory}

\subsection{Quantum polynomial functors}\label{sec:qfunctor}

We fix $\field$ a field, and $q \in \field^\times$ a non-zero scalar. We recall here the basic notions of quantum polynomial functors, see \cite{hong2017quantum} and \cite[section 2]{théo2025extgroupcategoryquantumpolynomial} for more details.

The Hecke algebra $\hecke{d}$ of the symmetric group $\symgroup{d}$ acts in a certain way on the family of vector spaces $V_n^{\otimes d}$ where $V_n = \field^n$. Let $\source{d}$ be the category with objects the $V_n^{\otimes d}$, for $n \geq 0$, and with morphisms $\hecke{d}$-linear maps.

\begin{definition}\label{def:qfunctor}
    A quantum polynomial functor of degree $d$ is a $\field$-linear functor 
    \[
        F : \source{d} \to \cible,
    \]
    where $\cible$ is the category of finite-dimensional $\field$-vector spaces. We denote by $\quantumfunctor{d}$ the category of quantum polynomial functors of degree $d$, with morphisms the natural transformations. The category $\quantumfunctor{} = \bigoplus_{d \geq 0} \quantumfunctor{d}$ is the category of quantum polynomial functors. 
\end{definition}
The category $\quantumfunctor{}$ is a braided monoidal category, with a notion of duality. When $q=1$, this category is equivalent to the category of strict polynomial functors of Friedlander and Suslin \cite{friedlander1997cohomology}. Unlike our category, the category of Friedlander and Suslin is defined as a subcategory of the category of functors $\cible \to \cible$. Hence, if $V$ is any finite-dimensional vector space and $F \in \polyfunctor{d}$, $F(V)$ is well defined. For our purposes, the reader can simply understand $F(V)$ as being  $F(V_n^{\otimes d})$ where $n$ is the dimension of $V$.

% One important detail : the category of strict polynomial functors $\polyfunctor{}$ are functors which accept in sources any vector space. For any vector space $V$ and any strict polynomial functor $F \in \polyfunctor{}$, the reader can understand $F(V)$ as being $F(V_n^{\otimes d})$ where $n$ is the dimension of $V$.

To simplify notations, for $F \in \quantumfunctor{d}$, we let $F(n) = F(V_n^{\otimes d})$ for $n \geq 0$. We give some examples of quantum polynomial functors to fix notations.
\begin{example}
    \begin{itemize}
        \item The identity functor $I \in \quantumfunctor{1}$, given by $I(n) = V_n$.
        \item The tensor power functors $\bigotimes^d = \underbrace{I \otimes \cdots \otimes I}_{d} \in \quantumfunctor{d}$, given by $\bigotimes^d(n) = V_n^{\otimes d}$.
        \item The symmetric tensor functors $S^d_q \in \quantumfunctor{d}$. When $q=1$, we denote them simply by $S^d$.
        \item The divided power functors $\Gamma^d_q$, which are dual to $S^d_q$.
        \item Tensor product of these functors. In particular, for $\mu = (\mu_1,...,\mu_n)$ an $n$-tuple of non-negative integers, we let $S^\mu_q = S^{\mu_1}_q \otimes \cdots \otimes S^{\mu_n}_q$ and $\Gamma^\mu_q = \Gamma^{\mu_1}_q \otimes \cdots \otimes \Gamma^{\mu_n}_q$.
    \end{itemize}
\end{example}

The algebra $S_q(n;d) = \End_{\hecke{d}}(V_n^{\otimes d})$ is called the $q$-Schur algebra. If $F \in \quantumfunctor{d}$ is a quantum polynomial functors, then $F(n)$ is a left $S_q(n;d)$-module. We use this structure to decompose $F(n)$ into weight spaces.

\begin{definition}\label{def:weight}
    The unit $1_d \in S_q(n;d)$ decomposes as a sum of orthogonal idempotents
    \[
        1_d = \sum_{\alpha \in \Omega(d,n)} \xi_\alpha
    \]
    where $\Omega(d,n)$ is the set of compositions of $d$ in $n$ parts ($n$-tuples of non-negative integers whose sum is $d$). We let $F_\alpha = \xi_\alpha \cdot F(n)$. Then $F_\alpha$ is called the weight space of $F$ associated to $\alpha$. If $F_\alpha \neq 0$, we say that $\alpha$ is a weight of $F$.
\end{definition} 

\begin{proposition}\label{prop:weighttensor}
    Let $F,G \in \quantumfunctor{}$. Then for any composition $\alpha$ in $n$ parts,
    \[
        (F \otimes G)_\alpha = \bigoplus_{\alpha^1 + \alpha^2 = \alpha} F_{\alpha^1} \otimes G_{\alpha^2}
    \]
    where the direct sum is over all pairs of compositions $(\alpha^1,\alpha^2)$ in $n$ parts whose sum (coordinate wise) is $\alpha$.
\end{proposition}
\begin{proof}
    See \cite[Proposition 2.11]{théo2025extgroupcategoryquantumpolynomial}.
\end{proof}

To finish this subsection, we give some properties of the functors $S^\alpha_q$. The important properties of these functors are given in the following proposition.

\begin{proposition}\label{prop:symmetric power property}
    The functors $S^\alpha_q$, where $\alpha$ runs over all compositions of $d$, form a cogenerating family of injective objects of $\quantumfunctor{d}$. Moreover, for any $\alpha \in \Omega(d,n)$,
     \[
         \dim((S^d_q)_\alpha) = 1,
     \]
     and for any $F \in \quantumfunctor{d}$,
     \[
        \dim \Hom_{\quantumfunctor{}}(F,S^\alpha_q) = \dim F_\alpha.
     \]
\end{proposition}
\begin{proof}
    See \cite[Proposition 3.4 et 3.5]{théo2025extgroupcategoryquantumpolynomial}.
\end{proof}

% \begin{proposition}\label{prop:cogenerator}
%     The functors $S^\alpha_q$, where $\alpha$ runs over all compositions of $d$, form a cogenerating family of injective objects of $\quantumfunctor{d}$.
% \end{proposition}
% \begin{proof}
%     See \cite[Proposition 3.4]{théo2025extgroupcategoryquantumpolynomial}.
% \end{proof}
% The two following propositions link the functors $S^\alpha_q$ with the notion of weight introduced in definition \ref{def:weight}. Recall that $\Omega(d,n)$ denote the set of compositions of $d$ in $n$ parts.
% \begin{proposition}\label{prop:weightSd}
%     For any $\alpha \in \Omega(d,n)$,
%     \[
%         \dim((S^d_q)_\alpha) = 1.
%     \]
% \end{proposition}
% \begin{proof}
%     See \cite[Proposition 3.5]{théo2025extgroupcategoryquantumpolynomial}.
% \end{proof}

% \begin{proposition}\label{prop:symmetric and weight}
%     For any quantum polynomial functor $F \in \quantumfunctor{}$, and any $\alpha \in \Omega(d,n)$,
%     \[
%         \dim \Hom_{\quantumfunctor{}}(F,S^\alpha_q) = \dim F_\alpha.
%     \]
% \end{proposition}
% \begin{proof}
%     See \cite[Proposition 3.4]{théo2025extgroupcategoryquantumpolynomial}.
% \end{proof}

\subsection{Quantum Frobenius twist}\label{sec:twist}

From now on and in the rest of this article, $q$ is a primitive $\ell$-th root of unity, with $\ell$ odd.
\begin{definition}\label{def:twist}
    There is a fully faithful and exact functor $-^{(1)_q} : \polyfunctor{d} \to \quantumfunctor{dl}$ called the quantum Frobenius twist.
\end{definition}
For more details, refer to \cite[Section 4]{théo2025extgroupcategoryquantumpolynomial}. Our goal is to compute the effect of this functor on Ext-groups. We will first need some properties of the quantum Frobenius twist. We first introduce some vocabulary.
\begin{definition}
    We say that $\ell$ divides a composition $\alpha = (\alpha_1,...,\alpha_n) \in \Omega(d,n)$ if $\ell$ divides each of the $\alpha_i$. Note that this implies that $\ell$ divides $d$. We also write $\ell \alpha = (\ell \alpha_1,...,\ell \alpha_n) \in \Omega(d\ell,n)$. Thus, $\ell$ divides $\alpha$ if and only if there exist a composition $\alpha'$ such that $\alpha = \ell \alpha'$.
\end{definition}

\begin{proposition}\label{prop:twist}
    For $F,G \in \polyfunctor{}$,
    \[
        (F \otimes G)^{(1)_q} \cong F^{(1)_q} \otimes G^{(1)_q}
    \]
    Moreover, for any composition $\alpha$, if $\ell$ does not divide $\alpha$, then $(F^{(1)_q})_\alpha = 0$, and if $\alpha = \ell \alpha'$, then $(F^{(1)_q})_\alpha = F_{\alpha'}$.
\end{proposition}
\begin{proof}
    See \cite[Proposition 4.3]{théo2025extgroupcategoryquantumpolynomial}.
\end{proof}
The following proposition gives the two first steps of finding an injective coresolution of the functors $S^{d(1)_q} \in \quantumfunctor{dl}$.
\begin{proposition}\label{prop:start}
    For $d \geq 0$, we have an injective morphism $\varphi_d : S^{d(1)_q} \to S^{d \ell}_q$, whose image is the kernel of the coproduct $\Delta^{(d\ell - 1,1)} : S^{d \ell}_q \to S^{d\ell - 1}_q \otimes S^1_q$ (see \cite[Proposition 3.2]{théo2025extgroupcategoryquantumpolynomial}).
\end{proposition}
\begin{proof}
    The morphism $\varphi_d$ is the morphism of \cite[Proposition 4.2]{théo2025extgroupcategoryquantumpolynomial}, the fact that its image is the kernel of $\Delta^{(d\ell-1,1)}$ is part of \cite[Theorem 5.6]{théo2025extgroupcategoryquantumpolynomial}, since $\Delta^{(d\ell-1,1)}$ is the first differential in the complex $\Omega^*_{d\ell}$.
\end{proof}

We conclude this subsection by a useful lemma.
\begin{lemma}\label{lem:Hom(frob,Sq)}
    Let $F \in \polyfunctor{d}$ be a strict polynomial functor, and let $\mu \in \Omega(d \ell,n)$ be a composition. Then we have isomorphisms, natural in $F$ :
    \[
        \Hom_{\quantumfunctor{}}(F^{(1)_q},S^\mu_q) \simeq \left \{ \begin{array}{cl}
            \Hom_{\polyfunctor{}}(F,S^{\frac{\mu}{\ell}}) & \mbox{if } \ell \mbox{ divides } \mu, \\
            0 & \mbox{otherwise.}
        \end{array} \right .
    \]
\end{lemma}
% \begin{lemma}\label{lem:zerofrob}
%     Let $F \in \polyfunctor{d}$ be a strict polynomial functor, and let $\mu \in \Omega(d \ell,n)$ be a composition. If $\ell$ does not divide $\mu$, then $\Hom_{\quantumfunctor{}}(F^{(1)_q},S^\mu_q) = 0$.
% \end{lemma}
% \begin{proof}
%     This is a direct consequence of propositions \ref{prop:twist} and \ref{prop:symmetric and weight}.
% \end{proof}

% \begin{lemma}\label{lem:isofrob}
%     Let $F \in \polyfunctor{d}$ be a strict polynomial functor. For any composition $\mu \in \Omega(d,n)$, we have isomorphisms
%     \[
%         \Hom_{\polyfunctor{}}(F,S^{\mu}) \cong \Hom_{\quantumfunctor{}}(F^{(1)_q},S^{\mu(1)_q}) \cong \Hom(F^{(1)_q},S^{\ell \mu}_q)
%     \]
%     natural in $F$.
% \end{lemma}
\begin{proof}
    If $\ell$ does not divide $\mu$, the result is a direct consequences of propositions \ref{prop:symmetric power property} and \ref{prop:twist}.

    Now suppose that $\ell$ divides $\mu$, say $\mu = \ell \nu$
    Since the quantum Frobenius twist is fully faithful, it induces an isomorphism 
    \[
        \Hom_{\polyfunctor{}}(F,S^{\nu}) \cong \Hom_{\quantumfunctor{}}(F^{(1)_q},S^{\nu(1)_q}).
    \]
    Moreover, we have an injective natural transformation $\varphi_\nu = \varphi_{\nu_1} \otimes \cdots \otimes \varphi_{\nu_n} : S^{\nu(1)_q} \to S^{\mu}_q$. Post-composition by $\varphi_\nu$ induces an injective morphism $\Hom_{\quantumfunctor{}}(F^{(1)_q},S^{\nu(1)_q}) \hookrightarrow \Hom(F^{(1)_q},S^{\mu}_q)$. Now, by propositions \ref{prop:symmetric power property} and \ref{prop:twist},
    \[
        \dim \Hom_{\polyfunctor{}}(F,S^{\nu}) = \dim \Hom(F^{(1)_q},S^{\mu}_q)
    \]
    and hence, the post-composition is in fact an isomorphism. The naturality in $F$ follows from the fact that it is a post-composition.
\end{proof}

\subsection{$\ell$-complexes}

Our quantum Troesch complexes will be $\ell$-complexes, so we give here everything the reader needs to know about $\ell$-complexes to understand this article. We work in an abelian monoidal $\field$-linear category, with biexact tensor product.

\begin{definition}
    A $\ell$-complex is a graded object $C^* = \bigoplus_{n \geq 0} C^n$ equipped with a $\ell$-differential, that is a morphism $d_C$ of degree $1$ such that $d_C^\ell = 0$.
\end{definition}
We can define the cohomology groups of $C$ as follows.
\begin{definition}
    Let $C$ be a $\ell$-complex. For all $s \in \{1,...,\ell-1\}$ and all $i \geq 0$, we define
    \[
        H^i_{[s]}(C) = \frac{\ker(\delta_C^s : C^i \to C^{i+s})}{\Image(\delta_C^{\ell-s} : C^{i-\ell+s} \to C^i)}.
    \]
    The $\ell$-complex $C$ is said to be acyclic if $H^i_{[s]}(C) = 0$ for all $s \in \{1,...,\ell-1\}$ and all $i \geq 0$. It is said to be a $\ell$-coresolution of $F$ if 
    \[
        H^i_{[s]}(C) = \left \{ \begin{array}{cl}
            F & \mbox{if } i=0, \\
            0 & \mbox{otherwise,}
        \end{array} \right .
    \]
    for all $s \in \{1,...,\ell-1\}$.
\end{definition}
We give some examples, which will be the basis of our computations of cohomology groups for more difficult $\ell$-complexes.
\begin{example}\label{example:lacyclic}
    Let $C$ be the $\ell$-complex of $\field$-vector spaces
    \[
        C : \underbrace{\field \xrightarrow{=} \field \xrightarrow{=} \cdots \xrightarrow{=} \field}_\ell .
    \]
    Then $C$ is acyclic.
\end{example}
\begin{example}\label{example:lcoresolution}
    Let $C$ be the $\ell$-complex of $\field$-vector spaces concentrated in degree $0$ with $C^0 = \field$. Then $C$ is a $\ell$-coresolution of $\field$.
\end{example}
Our computations of cohomology in theorem \ref{theo:cohomologyBd} will consist of using short exact sequences to bring back our computation to direct sums of the above complexes (with shifts in the degree). To use the short exact sequences, we will only need the following lemma.
\begin{lemma}\label{lem:short}
    Let
    \[
        0 \to C' \to C \to C'' \to 0
    \]
    a short exact sequence of $\ell$-complexes. Then
    \begin{itemize}
        \item if $C''$ and $C'$ are acyclic, then so is $C$ ;
        \item if $C'$ is acyclic, and $C''$ is a $\ell$-coresolution of $F$, then $C$ is also a $\ell$-coresolution of $F$.
    \end{itemize}
\end{lemma}
\begin{proof}
    Consider the ordinary complex $\underline{C}$ given by
    \[
        \underline{C} : C^{i - \ell + s} \to C^i \to C^{i+s} \to C^{i+\ell}
    \]
    (with $C^j = 0$ for $j<0$), and consider also $\underline{C'}$ and $\underline{C''}$ defined in a similar way. Then the short exact sequence of $\ell$-complexes gives a short exact sequence of ordinary complexes
    \[
        0 \to \underline{C'} \to \underline{C} \to \underline{C''} \to 0.
    \]
    The long exact sequence associated gives an exact sequence
    \[
        H^{i}_{[s]}(C') \to H^{i}_{[s]}(C) \to H^{i}_{[s]}(C'') \to H^{i+s}_{[l-s]}(C').
    \]
    The lemma follows.
\end{proof}

From a $\ell$-complex, we can construct several ordinary complexes as follows. Let $C$ be a $\ell$-complex. For any integer $s \in \{1,...,\ell-1\}$, we define 
\[
    C_{[s]} : C^0 \xrightarrow{d^s} C^s \xrightarrow{d^{\ell - s}} C^l \xrightarrow{d^s} C^{l+s} \xrightarrow{d^{\ell - s}} C^{2l} \xrightarrow{d^s} \cdots
\]
With this method, we will construct injective coresolutions of our quantum polynomial functors $\twist{(S^d)}$ from $\ell$-coresolutions of $\twist{(S^d)}$.

To obtain $\ell$-coresolutions of the functors $S^{\alpha(1)_q}$ from the $\ell$-coresolutions of $S^{\alpha_1(1)_q},...,S^{\alpha_n(1)_q}$, we will use the following construction.

\begin{definition}\label{def:ltensor}
    Let $C,D$ be two $\ell$-complexes. We define an $\ell$-complex $C \otimes D$ by
    \[
        (C \otimes D)^k = \bigoplus_{i+j=k} C^i \otimes D^j \quad \mbox{with} \quad d_{|C^i \otimes D^j} = (d_C \otimes 1) + q^{2i} (1 \otimes d_D).
    \]
\end{definition}
To see that it defines a $\ell$-complex, let $f,g$ be the endomorphisms of $C \otimes D$ given by
\[
    f_{|C^i \otimes D^j} = (d_C \otimes 1) \quad \mbox{and} \quad g_{|C^i \otimes D^j} = q^{2i} (1 \otimes d_D)
\]
so that $d_{C \otimes D} = f + g$. Then, since $g \circ f = q^2 f \circ g$, we have
\[
    d^\ell_{C \otimes D} = \sum_{k=0}^l \binom{\ell}{k}_{q^2} f^k g^{l-k} = f^\ell + g^\ell = 0
\]
(see \cite[Section 7.1 and Corollary 7.1.3]{parshall1991quantum}).

\begin{remark}
    A natural choice for the differential of $C \otimes D$ would be
    \[
        d_{|C^i \otimes D^j} = d_C \otimes 1 + q^i (1 \otimes d_D) \ .
    \]
    This choice is done e.g in \cite{duboisviolette2009tensorproductncomplexesgeneralization}. We use $q^{2i}$ instead of $q^i$ in order to ensure compatibility of differentials and products in lemma \ref{lem:prodB}.
\end{remark}

We have the following weak version of Künneth formula.
\begin{proposition}\label{prop:lkunneth}
    If $C$ and $D$ are $\ell$-coresolutions of $F$ and $G$ respectively, then $C \otimes D$ is an $\ell$-coresolution of $F \otimes G$.
\end{proposition}
\begin{proof}
    The definition of the tensor product is different, but there is no significant modification needed to adapt the proof of \cite[Theorem 2.3.1]{troesch2005resolution}.
\end{proof}

\subsection{Ext-computation using quantum Troesch complexes} \label{sec:comp}

Let $d \geq 0$. Consider the functor
\[
    B_d = \bigoplus_{\alpha \in \Omega(d,\ell)} S^\alpha_q.
\]
Define a grading on $B_d$ by letting $S^\alpha_q = S^{\alpha_1}_q \otimes \cdots \otimes S^{\alpha_\ell}_q$ be in degree $\deg(\alpha) = \sum_{i=1}^\ell (i-1)\alpha_i$. 

\begin{definition}\label{def:qTroesch}
    A quantum Troesch complex is an $\ell$-coresolution of $S^{d(1)_q}$ whose underlying graded functor is $B_{\ell d}$.
\end{definition}

Let $\alpha$ be a composition in $n$ parts. Then we define 
    \[
        B_{\alpha} = B_{\alpha_1} \otimes \cdots \otimes B_{\alpha_n}.
    \]
If $\ell$ divides $\alpha$, say $\alpha = \ell \alpha'$, and if each $B_{\alpha_i}$ is endowed with a structure of quantum Troesch complex, then $B_{\alpha}$ is endowed, by definition \ref{def:ltensor}, with a structure of $\ell$-coresolution of $S^{\alpha'(1)_q}$.
% \begin{definition}\label{def:qTroesch}
%     Consider the graded functor
%     \[
%         B_d^* = \bigoplus_{\alpha \in \Omega(d,\ell)} S^\alpha_q
%     \]
%     with $S^\alpha_q$ in degree $\deg(\alpha) = \sum_i (i-1)\alpha_i$. Assume that, for $d \geq 0$, $B_{d\ell}$ can be endowed with an $\ell$-differential making it a $\ell$-coresolution of $S^{d(1)_q}$. Let $\alpha$ be a composition in $n$ parts. Then we define 
%     \[
%         B_{\alpha} = B_{\alpha_1} \otimes \cdots \otimes B_{\alpha_n}.
%     \]
%     and $B_{\ell \alpha}$ is endowed with the structure of an $\ell$-complex as in \ref{def:ltensor}, making it an $\ell$-coresolution by proposition \ref{prop:lkunneth}.
% \end{definition}
\begin{remark}\label{rmk:choiceofthestart}
    Since $\Hom_{\quantumfunctor{}}(S^{d(1)_q},S^{d \ell}_q)$ is one-dimensional by proposition \ref{prop:symmetric power property} and \ref{prop:twist}, the coresolution map $S^{d(1)_q} \to B_{d \ell}$ can be chosen to be $\varphi_d$, and hence the map $S^{\alpha(1)_q} \to B_{\ell \alpha}$ can be chosen to be $\varphi_\alpha$. We will assume that it is the choice we have made in the rest of this article.
\end{remark}
For now, we assume that we have succeeded in building quantum Troesch complexes. We can already use it to compute $\Ext$-group. To this end, we recall the following parametrization of strict polynomial functors, defined in \cite[Section 2.5]{touze2012troesch}.
\begin{definition}\label{def:parametre}
    Let $F$ be a strict polynomial functor, and $V$ be a finite-dimensional graded vector space. We define a graded strict polynomial functor $F_V$ by letting $F_V(W) = F(\bigoplus_i V^i\otimes W)$, and with graduation defined as follows. Let the multiplicative group $\mathbb{G}_m$ acts on each $V^i$ with weight $i$. This endows $F_V(W)$ with a rational action of $\mathbb{G}_m$, and hence $F_V(W)$ has a weight space decomposition $F_V(W) = \bigoplus_i F_V(W)^i$. Then $(F_V)^i(W) = F_V(W)^i$.
\end{definition}
\begin{remark}\label{rmk:parametrenaturality}
    If $f : F \to G$ is a natural transformation, then $f_V = f(V^* \otimes -) : F_V \to G_V$ is also a natural transformation.
\end{remark}

This parametrization satisfies the following properties, which can be found in \cite[Lemma 2.8]{touze2012troesch}.
\begin{lemma}\label{lem:paraprop}
    Let $V$ be a finite-dimensional graded vector space. Then the functor $F \mapsto F_V$ is exact and commutes with tensor products : for all strict polynomial functors $F,G$,
    \[
        (F \otimes G)_{V} = F_V \otimes G_V
    \]
    as graded functors.
\end{lemma}
The finite-dimensional vector space useful for our purposes is the following one.
\begin{definition}\label{def:E1}
    We set $E$ to be the graded vector space given by
    \[
        E^{k} = \left \{ \begin{array}{cl}
             \field & \mbox{if } k=0,2, \dots,2(\ell - 1), \\
             0 & \mbox{otherwise.} 
        \end{array} \right .
    \]
\end{definition}
The graded vector space $E$ is in fact isomorphic to $\Ext_{\quantumfunctor{}}^*(\twist{I},\twist{I})$, see \cite[Theorem 6.3]{théo2025extgroupcategoryquantumpolynomial}.

\begin{examples}
    One can show the following isomorphisms of graded functors.
    \begin{itemize}
        \item $\bigotimes^d_E \cong E^{\otimes d} \otimes \bigotimes^d$.
        \item $S^d_E \cong \bigoplus_{\alpha \in \Omega(d,l)} S^\alpha$ with $S^\alpha$ in degree $2\deg(\alpha)$.
    \end{itemize}
\end{examples}

\begin{proposition}\label{prop:firstiso}
    Let $F \in \polyfunctor{d}$ be an homogeneous strict polynomial functor. Then for any composition $\mu \in \Omega(d,n)$, we have a graded isomorphism, natural in $F$,
    \[
        \Ext^*_{\quantumfunctor{}}(\twist{F}, \twist{(S^\mu)}) \cong \Hom_{\polyfunctor{}}(F,(S^\mu)_{E}).
    \]
\end{proposition}
\begin{proof}
    Consider the injective coresolution $B_{\ell \alpha [1]}$ of $S^{\alpha(1)_q}$. It is given by
    \[
        B^{2k}_{\ell \alpha [1]} = \bigoplus_{\substack{\beta^i \in \Omega(\ell \alpha_i,\ell)\\\sum_i \deg(\beta^i) = \ell k}} \left (\bigotimes_{i=1}^n S^{\beta^i}_q \right ) \qquad \mbox{and} \qquad B^{2k+1}_{\ell \alpha [1]} = \bigoplus_{\substack{\beta^i \in \Omega(\ell \alpha_i,\ell)\\\sum_i \deg(\beta^i) = \ell k+1}} \left (\bigotimes_{i=1}^n S^{\beta^i}_q \right ).
    \]
    By lemma \ref{lem:Hom(frob,Sq)}, $\Hom(F^{(1)_q}, B^{2k+1}_{\ell \alpha [1]}) = 0$. Thus,
    \[
        \Ext^*_{\quantumfunctor{}}(F^{(1)_q}, S^{\alpha(1)_q}) \cong \left \{
        \begin{array}{cl}
            \Hom(F^{(1)_q},\ell(B_{\ell \alpha_1}, B_{\ell \alpha_2}, ... , B_{\ell \alpha_n})^*) & \mbox{if } k \mbox{ is even,} \\
            0 & \mbox{otherwise.} 
        \end{array} \right .
    \]
    where 
    \[
        \ell(B_{\ell \alpha_1}, B_{\ell \alpha_2}, ... , B_{\ell \alpha_n})^* = \bigoplus_{t_1,...,t_n} B_{\ell \alpha_1}^{\ell t_1} \otimes B_{\ell \alpha_2}^{\ell t_2} \otimes \cdots \otimes B_{\ell \alpha_n}^{\ell t_n}
    \]
    with $ B_{\ell \alpha_1}^{\ell t_1} \otimes B_{\ell \alpha_2}^{\ell t_2} \otimes \cdots \otimes B_{\ell \alpha_n}^{\ell t_n}$ in degree $2(t_1 + \cdots + t_n)$. By the exponential property (see \cite[Section 2.2]{touze2012troesch}), 
    \[
        S^d_{E} = \bigoplus_{\mu \in \Omega(d,\ell)} S^\mu
    \]
    with $S^\mu$ in degree $2\sum_i (i-1)\mu_i$. Hence, by lemma \ref{lem:Hom(frob,Sq)}, we have a graded isomorphism
    \begin{align*}
        \Hom_{\polyfunctor{}}(F,S^\alpha_{E}) 
        & \cong \bigoplus_{\beta^i \in \Omega(\alpha_i,\ell)} \Hom_{\polyfunctor{}}(F,S^{\beta^1} \otimes \cdots \otimes S^{\beta^n}) \\
        & \cong \bigoplus_{\beta^i \in \Omega(\alpha_i,\ell)} \Hom_{\quantumfunctor{}}(F^{(1)_q},S^{\ell \beta^1} \otimes \cdots \otimes S^{\ell \beta^n}) \\
        & \cong \Hom(F^{(1)_q},\ell(B_{\ell \alpha_1}, B_{\ell \alpha_2}, ... , B_{\ell \alpha_n})).
    \end{align*}
    whence the result.
\end{proof}
% We will construct a second isomorphism $\Ext^*_{\quantumfunctor{}}(\twist{F}, \twist{(S^\mu)}) \cong \Hom_{\polyfunctor{}}(F,(S^\mu)_{E})$, which will be natural in $F$ and in $S^\mu$ and which will satisfy another good property. The isomorphism of proposition \ref{prop:firstiso} is important to show the following corollary.
\begin{corollary}\label{cor:exactitude}
    For any composition $\mu$, the following functors are exact :
    \[
        F \mapsto \Ext^*_{\quantumfunctor{}}(\twist{F},\twist{(S^{\mu})}) \quad \mbox{and} \quad G \mapsto \Ext^*_{\quantumfunctor{}}(\twist{(\Gamma^{\mu})},\twist{G})\ .
    \]
\end{corollary}
\begin{proof}
    Using duality, it suffices to prove the exactness of the first functor. For that, we just use the isomorphism (natural in $F$) of proposition \ref{prop:firstiso} and the fact that $(S^\mu)_{E}$ is injective.
\end{proof}
In fact, we need the quantum Troesch complexes solely to show corollary \ref{cor:exactitude}. The remainder of this section is devoted to constructing another isomorphism
\[
    \Ext_{\quantumfunctor{}}^*(\twist{F}, \twist{(S^\mu)}) \simeq \Hom_{\quantumfunctor{}}(F,(S^\mu)_E) \ ,
\]
and the twisting spectral sequence. The new isomorphism will be compatible with the cup product, whose definition is recalled just below.
\begin{definition}
    Let $F_1,F_2,G_1,G_2$ be quantum polynomial functors, and let $P,Q$ be projective resolutions of $F_1$ and $F_2$ respectively. Then $P \otimes Q$ is a projective resolution of $F_1 \otimes F_2$, and the natural map
    \[
        \Hom_{\quantumfunctor{}}(P,G_1) \otimes \Hom_{\quantumfunctor{}}(Q,G_2) \to \Hom(P \otimes Q, G_1 \otimes G_2)
    \]
    is a complex morphism. Thus, it induces a map
    \[
        \Ext_{\quantumfunctor{}}(F_1,G_1) \otimes \Ext_{\quantumfunctor{}}(F_2,G_2) \to \Ext_{\quantumfunctor{}}(F_1 \otimes F_2, G_1 \otimes G_2), \quad c_1 \otimes c_2 \mapsto c_1 \cup c_2
    \]
    called the cup product.
\end{definition}
We first construct the new isomorphism for $F$ and $S^\mu$ tensor power functors.
\begin{theorem}\label{theo:tenseurcase}
    There exist a family of graded isomorphisms
    \[
        \theta(\tensorfunctor{d},\tensorfunctor{d}) : \Ext^*_{\quantumfunctor{}}(\tensorfunctor{d(1)_q},\tensorfunctor{d(1)_q}) \to \Hom_{\polyfunctor{}}(\tensorfunctor{d}, (\tensorfunctor{d})_{E})
    \]
    for $d \geq 0$, which are natural in $\tensorfunctor{d},\tensorfunctor{d}$, and compatible with products :
    \[
        \theta(\tensorfunctor{d_1+d_2},\tensorfunctor{d_1+d_2})(c_1 \cup c_2) = \theta(\tensorfunctor{d_1},\tensorfunctor{d_1})(c_1) \otimes \theta(\tensorfunctor{d_2},\tensorfunctor{d_2})(c_2).
    \]
\end{theorem}
\begin{proof}
    According to \cite[Lemma 2.1]{touze2012troesch}, $\Hom_{\polyfunctor{}}(\tensorfunctor{d},\tensorfunctor{d}) \cong \field \symgroup{d}$. Thus, the naturality means that $\theta(\tensorfunctor{d},\tensorfunctor{d})$ is a morphism of $\field \symgroup{d} - \field \symgroup{d}$-bimodule. 
    
    As a $\field \symgroup{d} - \field \symgroup{d}$-bimodule, $\Ext^*_{\quantumfunctor{}}(\tensorfunctor{d(1)_q},\tensorfunctor{d(1)_q})$ is isomorphic to $E^{\otimes d} \otimes \field \symgroup{d}$ by \cite[Theorem 6.6 and corollary 6.7]{théo2025extgroupcategoryquantumpolynomial}, where the action is defined by
    \[
        \sigma_1(v_1 \otimes \cdots \otimes v_d \otimes \omega)\sigma_2 = v_{\sigma_1(1)} \otimes \cdots \otimes v_{\sigma_1(d)} \otimes (\sigma_1 \omega \sigma_2).
    \]
    Meanwhile, $(\tensorfunctor{d})_E \cong E^{\otimes d} \otimes \tensorfunctor{d}$, and hence,
    \[
        \Hom_{\polyfunctor{}}(\tensorfunctor{d},(\tensorfunctor{d})_E) \cong E^{\otimes d} \otimes \Hom_{\polyfunctor{}}(\tensorfunctor{d},\tensorfunctor{d}) \cong E^{\otimes d} \otimes \field \symgroup{d}
    \]
    and this isomorphism is also an isomorphism of $\field \symgroup{d} - \field \symgroup{d}$-bimodule. The composition of the isomorphisms gives the following isomorphisms of $\field \symgroup{d} - \field \symgroup{d}$-bimodule :
    \[
        \begin{array}{rcl}
             \theta(\tensorfunctor{d},\tensorfunctor{d}) : \Ext^*_{\quantumfunctor{}}(\tensorfunctor{d(1)_q},\tensorfunctor{d(1)_q}) & \to & \Hom_{\polyfunctor{}}(\tensorfunctor{d}, (\tensorfunctor{d})_{E})  \\
            (v_1 \cup \cdots \cup v_d) \cdot \sigma & \mapsto & (x \mapsto v_1 \otimes \cdots \otimes v_d \otimes \sigma(x)).
        \end{array}
    \]
    The compatibility with the product can then be checked directly.
\end{proof}

We can now do the general case.
\begin{theorem}\label{theo:naturaliso}
    Let $F$ be a strict polynomial functor, and $\mu$ a composition. We have a graded isomorphism
    \[
        \theta(F,S^\mu) : \Ext^*_{\quantumfunctor{}}(\twist{F},\twist{(S^\mu)}) \to \Hom_{\polyfunctor{}}(F,(S^\mu_{E})^*)
    \]
    natural in $F$ and in $S^\mu$. Moreover, if $F_1,F_2$ are strict polynomial functors and $\mu^1,\mu^2$ are compositions, then, for any
    \[
        c_1 \in \Ext^*_{\quantumfunctor{}}(\twist{F_1},\twist{(S^{\mu^1})}) \quad \mbox{and} \quad c_2 \in \Ext^*_{\quantumfunctor{}}(\twist{F_2},\twist{(S^{\mu^2})}),
    \]
    we have
    \[
        \theta(F_1 \otimes F_2,S^{\mu^1} \otimes S^{\mu^2})(c_1 \cup c_2) = \theta(F_1,S^{\mu^1})(c_1) \otimes \theta(F_2,S^{\mu^2})(c_2).
    \]
\end{theorem}
\begin{proof}
    The proof is similar to the proof of \cite[Theorem 5.6]{touze2012troesch}. This only uses the case $F=G=\bigotimes^d$, which is covered by theorem \ref{theo:tenseurcase}, and the exactness of the functors of corollary \ref{cor:exactitude}.
\end{proof}

Finally, we present the twisting spectral sequences.
\begin{theorem}\label{theo:spectralsequence}
    Let $F,G$ be strict polynomial functors. Then there is a first quadrant spectral sequence, natural in $F$ and $G$,
    \[
        E_2^{s,t}(F,G) = \Ext^s_{\polyfunctor{}}(F,G_{E}^t) \implies \Ext_{\quantumfunctor{}}^{s+t}(\twist{F},\twist{G}).
    \]
    Moreover, if $F_1,F_2,G_1,G_2$ are strict polynomial functors, there is a pairing of spectral sequences
    \[
        E^{*,*}(F_1,G_1) \otimes E^{*,*}(F_2,G_2) \to E^{*,*}(F_1 \otimes F_2, G_1 \otimes G_2) 
    \]
    (see \cite[3.9]{benson1991representations} for a definition) which coincides with the cup product on the second page and on the abutment.
\end{theorem}
\begin{proof}
    It can be proved exactly as in \cite[Theorem 7.1]{touze2012troesch} using theorem \ref{theo:naturaliso}.
\end{proof}
\begin{corollary}\label{cor : collapse imply conjecture}
    Let $F,G$ be strict polynomial functors. Suppose that the twisting spectral sequence collapses at the second page. Then
    \[
        \Ext^k_{\quantumfunctor{}}(F^{(1)_q},G^{(1)_q}) \simeq \bigoplus_{i+j=k} \Ext^i_{\polyfunctor{}}(F,G^j_{E}) \ .
    \]
    In particular, this holds if $\Ext^s_{\polyfunctor{}}(F,G_{E}^t) = 0$ when $s$ is odd. 
\end{corollary}
For example, the hypothesis holds for the following choice of $F$ and $G$ :
% If, for some strict polynomial functor $F,G$, the twisting spectral sequence collapses at the second page, then this proves for the conjecture \ref{conj} for these particular $F$ and $G$. It can happens for lacunary reason, for example if $\Ext^s_{\polyfunctor{}}(F,G_{E}^t) = 0$ when $s$ is odd (since $G^t_E = 0$ when $t$ is odd), or when $s \neq 0$. This happens in the following cases :
\begin{itemize}
    \item when $F$ is projective, or dually if $G$ is injective;
    \item when $F$ is a Schur functor and $G$ a Weyl functor;
    \item when $F$ is obtained from a projective by the classical Frobenius twist and $G$ is obtained from any functor by the classical Frobenius twist, or dually with $G$ obtained from an injective by the classical Frobenius twist, and $F$ from any functor.
\end{itemize}

\section{Construction of the Troesch complex in the case $\ell = 3$}

\subsection{Exponential functors}

To construct and compute the homology of our complexes, we will define more structure on the underlying graded functor of the Troesch complexes. We will define a product, which will have the special property of making the family $(B_d)_{d \geq 0}$ an exponential functor.

\begin{definition}\label{def : algebra in the category of quantum polynomial functors}
    A $\quantumfunctor{}$-algebra is a family of quantum polynomial functors $F = (F_d)_{d \geq 0}$, with $F_d \in \quantumfunctor{d}$, endowed with :
    \begin{itemize}
        \item a family $\mu_F = (\mu_F^{d_1,d_2})_{d_1,d_2 \geq 0}$ of natural transformations $\mu_F^{d_1,d_2} : F_{d_1} \otimes F_{d_2} \to F_{d_1+d_2}$, called the product,
        \item a natural transformation $\eta_F : \field \to F_0$ (where $\field$ is the constant functor $\in \quantumfunctor{0}$), called the unit,
    \end{itemize}
    such that
    \begin{itemize}
        \item $\mu_F$ is associative : for all $d_1,d_2,d_3 \geq 0$,
        \[
            \mu_F^{d_1+d_2,d_3} \circ (\mu_F^{d_1,d_2} \otimes \id_{F_{d_3}}) = \mu_F^{d_1,d_2+d_3} \circ (\id_{F_{d_1}} \otimes \mu_F^{d_2,d_3});
        \]
        \item $\eta_F$ is an unit for $\mu_F$ : for all $d \geq 0$, $\mu^{d,0} \circ (\id_{F_d} \otimes \eta_F) : F_d \otimes \field \to F_d$ and $\mu^{0,d} \circ (\eta_F \otimes \id_{F_d}) : \field \otimes F_d \to F_d$ are the natural isomorphisms.
    \end{itemize}
\end{definition}
For example, $(\bigotimes^d)_{d \geq 0}$, $S_q = (S^d_q)_{d \geq 0}$, $\Gamma_q = (\Gamma^d_q)_{d \geq 0}$ and $\Lambda_q = (\Lambda^d_q)_{d \geq 0}$ can be endowed with products and units such that they become $\quantumfunctor{}$-algebra (see \cite[section 3]{théo2025extgroupcategoryquantumpolynomial}). Tensor products of $\quantumfunctor{}$-algebras can also be endowed with a $\quantumfunctor{}$-algebra structure.
\begin{definition}\label{prop : tensor product of algebra in the category of quantum polynomial functor}
    Let $A = (A_d)_{d \geq 0}$, $B = (B_d)_{d \geq 0}$ be two $\quantumfunctor{}$-algebra. We let 
    \[
        A \otimes B = \left (\bigoplus_{d_1+d_2 = d} A_{d_1} \otimes B_{d_2} \right )_{d \geq 0}.
    \]
    Then $A \otimes B$ can be endowed with a structure of $\quantumfunctor{}$-algebra, with product $\mu_{A \otimes B}$ given by the composition
    \[
        (A_{d_1} \otimes B_{d_2}) \otimes (A_{d_3} \otimes B_{d_4}) \xrightarrow{1 \otimes R_{B_{d_2},A_{d_3}} \otimes 1} (A_{d_1} \otimes A_{d_3}) \otimes (B_{d_2} \otimes B_{d_4}) \xrightarrow{\mu_A \otimes \mu_B} A_{d_1+d_3} \otimes B_{d_2+d_4}\ ,
    \]
    where $R$ denote the braiding of $\quantumfunctor{}$, and $\eta_{A \otimes B} = \eta_A \otimes \eta_B$ (identifying $\field$ and $\field \otimes \field$).
\end{definition}
The fact that this defines a $\quantumfunctor{}$-algebra follows directly from the property of the braiding, see \cite[proposition 2.8]{théo2025extgroupcategoryquantumpolynomial}.

All our examples of $\quantumfunctor{}$-algebra until now except $(\bigotimes^d)_d$ are in fact exponential functors. To define what it means, we introduce some maps.
\begin{notation}\label{notation : inclusions}
    For $d,n,m \geq 0$, we have an isomorphism of right $\hecke{d}$-modules
    \[
        V_{n+m}^{\otimes d} \simeq \bigoplus_{d_1+d_2 = d} \Induction_{\hecke{d_1 \otimes} \hecke{d_2}}^{\hecke{d}}(V_n^{\otimes d_1} \otimes V_m^{\otimes d_2}).
    \]
    Considering the case $d_1 = d$ and the case $d_2 = d$, we obtain injections $V_n^{\otimes d} \to V_{n+m}^{\otimes d}$ and $V_m^{\otimes d} \to V_{n+m}^{\otimes d}$, denoted $\iota_d(n,m)$ and $\iota_d'(n,m)$ respectively.
\end{notation}
\begin{definition}\label{def : exponential functor}
    Let $E = (E_d)_{d \geq 0}$ be a $\quantumfunctor{}$-algebra. Suppose that, for all $n,m \geq 0$ and $d \geq 0$, the map
    \[
        \exp^d_E(n,m) : \bigoplus_{d_1+d_2=d} E_{d_1}(n) \otimes E_{d_2}(m) \to \bigoplus_{d_1+d_2=d} E_{d_1}(n+m) \otimes E_{d_2}(n+m) \xrightarrow{\mu_E} E_d(n+m)
    \]
    called the exponential map, is an isomorphism, where the first map is given by $E_{d_1}(\iota_{d_1}(n,m)) \otimes E_{d_2}(\iota'_{d_2}(n,m))$.
\end{definition}
For more details on exponential functors, see \cite{touze2021structure}.

For $S_q$, the fact that it is an isomorphism can be seen as a consequences of the isomorphism in notation \ref{notation : inclusions}. For $B$, it is then a direct consequences of the following facts.
\begin{proposition}
    Let $A,B$ be two exponential functor. Then $A \otimes B$, with the products of definition \ref{prop : tensor product of algebra in the category of quantum polynomial functor}, is an exponential functor.
\end{proposition}
\begin{proof}
    The map $\exp_{A \otimes B}^d$ is a composition of braidings and maps $\exp_A^{d_1}$ and $\exp_B^{d_2}$, which are all isomorphisms. Hence $\exp_{A \otimes B}^d$ is an isomorphism.
\end{proof}
In the case of $B$, we also have on each $B_d$ a graduation. It can be verified that the products and the exponential we defined are graded maps. We summarize these informations in the following proposition.
\begin{proposition}\label{prop : structure of the underlying}
    The family $(B_d)_{d \geq 0}$ is a graded exponential functor.
\end{proposition}

\begin{warning}
    The product that we will use on $(B_d)_{d \geq 0}$ to define the quantum Troesch complex will not be the product defined in \ref{prop : tensor product of algebra in the category of quantum polynomial functor} but a twisted version of it. With this twisted version, $(B_d)_{d \geq 0}$ will remains an exponential functor.
\end{warning}
% Here, we don't precise the products, but it can be constructed from the definition \ref{prop : tensor product of algebra in the category of quantum polynomial functor}. The problem is that this will not be the good product for our purpose. We will need to twist it a little for it to be compatible with the $\ell$-differential, see lemma \ref{lem:prodB}. 

\subsection{Definition of the $3$-differential}

We now construct quantum Troesch complexes for the case $\ell = 3$. More precisely, we need to construct a $3$-differential $\delta : B_{3d} \to B_{3d}$, and to prove that the resulting $3$-complexes are $3$-coresolutions of $\twist{(S^d)}$.
\begin{definition}\label{def : Troesch pour l=3}
    Let $d \geq 0$. Recall that
    \[
        B_d = \bigoplus_{\alpha \in \Omega(d,3)} S^{\alpha_1}_q \otimes S^{\alpha_2}_q \otimes S^{\alpha_3}_q
    \]
    with $S^{\alpha_1}_q \otimes S^{\alpha_2}_q \otimes S^{\alpha_3}_q$ in degree $\alpha_2 + 2 \alpha_3$. Let $\delta_1 : B_d \to B_d$, $\delta_2 : B_d \to B_d$ be the degree $1$ maps given on $S^{\alpha_1}_q \otimes S^{\alpha_2}_q \otimes S^{\alpha_3}_q$ by
    \[
        \delta_{1|S^\alpha_q} = (1 \otimes \mu^{(1,\alpha_2)} \otimes 1)(\Delta^{(\alpha_1-1,1)} \otimes 1 \otimes 1) \quad \mbox{and} \quad \delta_{2|S^\alpha_q} = (1 \otimes 1 \otimes \mu^{(1,\alpha_3)})(1 \otimes \Delta^{(\alpha_2-1,1)} \otimes 1)
    \] 
    where $\mu^{(i,j)} : S^i_q \otimes S^j_q \to S^{i+j}_q$ is the product, and $\Delta^{(i,j)} : S^{i+j}_q \to S^i_q \otimes S^j_q$ is the coproduct (see \cite[Section 3]{théo2025extgroupcategoryquantumpolynomial}). We set $\delta : B_d \to B_d$ to be the degree $1$ map given on $S^{\alpha_1}_q \otimes S^{\alpha_2}_q \otimes S^{\alpha_3}_q$ by
    \[
        \delta_{|S^{\alpha}_q} = \delta_1 + (q^2)^{\alpha_2} \delta_2.
    \]
\end{definition}

As noted earlier, in order to prove that $\delta$ is a $3$-differential, we will twist the product obtained with proposition \ref{prop : tensor product of algebra in the category of quantum polynomial functor} by identifying $B$ with $(S^*_q)^{\otimes 3}$, so that $B$ is a graded differential $\quantumfunctor{}$-algebra (for details on the relation of compatibility, see lemma \ref{lem:prodB}).
\begin{notation}
    To write the definition of our product, we introduce several notations to express natural transformations $S^{\alpha}_q \to S^\beta_q$ in terms of elements. Let $d_1,d_2 \geq 0$ and $d = d_1 + d_2$.
    \begin{itemize}
        \item The products : $\mu = \mu^{(d_1,d_2)} :S^{d_1}_q \otimes S^{d_2}_q \to S^d_q$. We will write $\mu(x \otimes y) = xy$. These natural transformation induces a graded algebra structure on $S^*_q(n)$ for $n \geq 0$. In particular, $(xy)z = x(yz)$, and we can remove the parenthesis. The products are surjectives.
        \item The coproducts : $\Delta^{(d_1,d_2)} : S^d_q \to S^{d_1}_q \otimes S^{d_2}_q$. We will write 
        \[
            \Delta^{(d_1,d_2)}(x) = \sum_{x : (d_1,d_2)} x' \otimes x''.
        \]
        When $d_2 = 1$, we just write $\Delta^{(d-1,1)}(x) = \sum x' \otimes x''$ to simplify. They induces a graded coalgebra structure on $S^*_q(n)$. In particular, we can define more general maps $\Delta^{(d_1,d_2,...,d_k)} : S^{d_1 + d_2 + \cdots + d_k}_q \to S^{d_1}_q \otimes S^{d_2}_q \otimes \cdots \otimes S^{d_k}_q$, and we write
        \[
            \Delta^{(d_1,d_2,d_3)}(x) = \sum_{x: (d_1,d_2,d_3)} x' \otimes x'' \otimes x'''
        \]
        and so on.
        \item The switch : $R^{(d_1,d_2)} : S^{d_1}_q \otimes S^{d_2}_q \to S^{d_2}_q \otimes S^{d_1}_q$, given by the braiding. We write
        \[
            R^{(d_1,d_2)}(x \otimes y) = \sum r(y) \otimes r(x).
        \]
        Note that $r(y)$ depends on $r(x)$ and vice versa. More generally, we put indexes to explicit the order of the switches we apply. For example :
        \[
            \sum r_2(r_1(z)) \otimes r_2(x) \otimes r_1(y)
        \]
        means that, starting from $x \otimes y \otimes z$, we first switch $y$ and $z$, then $x$ and $z$. Note that
        \begin{equation}
            \sum r_1(y_1) \otimes r_1(x_1) \otimes r_2(y_2) \otimes r_2(x_2) = \sum r_2(y_1) \otimes r_2(x_1) \otimes r_1(y_2) \otimes r_1(x_2)
        \end{equation}
        with these notations. The switches are isomorphisms.
    \end{itemize}
    Finally, if $x \in S^d_q$, we write $|x|=d$.
\end{notation}
We have a number of relations expressing compatibility between products, coproducts and switches, that we list here. They can all be found in \cite[Section 3]{théo2025extgroupcategoryquantumpolynomial}.
\begin{align}
    \sum r(y)r(x) & = q^{|a||b|} xy \ ; \label{eq:detwist} \\
    \sum r_1(y)r_2(z) \otimes r_2(r_1(x)) & = \sum r(yz) \otimes r(x) \ ; \label{eq:switchprodbc} \\
    \sum r_2(r_1(z)) \otimes r_2(x) r_1(y) & = \sum r(z) \otimes r(xy) \ ; \label{eq:switchprodab} \\
    \sum_{x:(d_1,d_2)} r_2(r_1(y)) \otimes r_2(x') \otimes r_1(x'') & = \sum_{r(x):(d_1,d_2)} r(y) \otimes r(x)' \otimes r(x)'' \ ; \label{eq:switchcoproda} \\
    \sum_{y:(d_1,d_2)} r_1(y') \otimes r_2(y'') \otimes r_2(r_1(x)) & = \sum_{r(y):(d_1,d_2)} r(y)' \otimes r(y)'' \otimes r(x) \ ; \label{eq:switchcoprodb} \\
    \sum_{xy : (d_1,d_2)} (xy)' \otimes (xy)'' & = \sum_{\substack{x:(i,j) \\ y:(d_1-i,d_2-j)}} q^{j(d_1-i)} x' r(y') \otimes r(x'') y'' \ . \label{eq:dumbbell}
\end{align}
The relation (\ref{eq:dumbbell}) in the special case $d_2=1$ can be rewritten as
\begin{equation}\label{eq:dumbbell1}
    \sum(xy)' \otimes (xy)'' = \left ( \sum xy' \otimes y'' \right ) + q^{|y|} \left ( \sum x'r(y) \otimes r(x'') \right ).
\end{equation}
With these notations, $\delta$ can be rewritten as
\[
    \delta(x_1 \otimes x_2 \otimes x_3) = \left (\sum x_1' \otimes x_1''x_2 \otimes x_3 \right ) + q^{2|x_2|} \left (\sum x_1 \otimes x_2' \otimes x_2''x_3 \right ) .
\]
We now define the products. 
% To prove that it is a $3$-differential, we first define surjectives graded maps $\mu^{(d_1,d_2)}_B : B_{d_1} \otimes B_{d_2} \to B_{d_1+d_2}$ which commutes with $\delta$. The surjectivity can then be used to prove that $\delta^3 = 0$.

\begin{lemma}\label{lem:prodB}
    Let $d_1,d_2 \geq 0$ and $d = d_1 + d_2$. Define $\mu^{(d_1,d_2)}_B : B_{d_1} \otimes B_{d_2} \to B_{d_1+d_2}$ by
    \[
        \mu^{(d_1,d_2)}_B(a \otimes b) = q^{(|a_2| + |a_3|)|b_1|} \sum a_1 r_2(r_1(b_1)) \otimes r_2(a_2) r_3(b_2) \otimes r_3(r_1(a_3))b_3
    \]
    where $a = a_1 \otimes a_2 \otimes a_3$ and $b = b_1 \otimes b_2 \otimes b_3$. 
    Then $\mu_B$ defines an associative product on $B_*$, making it a graded differential $\quantumfunctor{}$-algebra and a graded exponential functors. In particular, $\mu_B^{(d_1,d_2)}$ is surjective and 
    \[
        \mu^{(d_1,d_2)}_B \circ \delta_{B_{d_1} \otimes B_{d_2}} = \delta \circ \mu_B^{(d_1,d_2)}    
    \]
    where $\delta_{B_{d_1} \otimes B_{d_2}}$ is defined as in definition \ref{def:ltensor}. 
\end{lemma}
As we said before, this product is not the most natural one. It would be more natural to do the same product without $q^{(|a_2| + |a_3|)|b_1|}$. But we would not have compatibility with the differential. It is the principal difficulty to find a differential and a product working together, and it is the reason the author was not able to construct Troesch complexes for $\ell > 3$.
\begin{proof}
    We will only detail the compatibility with the differential, which is the difficult (and most useful) part. So we start by rushing the rest of the facts.
    
    The $\mu_B^{(d_1,d_2)}$ are compositions of natural transformations (switches and products) and hence is natural. The associativity comes from several applications of (\ref{eq:switchprodbc}), (\ref{eq:switchprodab}) and the associativity of the products of $S^*_q$. The unit is simply $1 \otimes 1 \otimes 1$. The fact that it is graded is simple to verify. For the exponential property, since we only have made some multiplication by non-zero scalars, the image of the exponential map does not change from the one for the product of proposition \ref{prop : structure of the underlying}. Hence, the new exponential map is also an isomorphism.
    
    We now prove the compatibility with the differential. We compute
    \begin{align*}
        \delta & \mu^{(d_1,d_2)}_B (a \otimes b) = \\ & q^{(|a_2| + |a_3|)|b_1|} \sum (a_1 r_2(r_1(b_1)))' \otimes (a_1 r_2(r_1(b_1)))'' r_2(a_2) r_3(b_2) \otimes r_3(r_1(a_3))b_3 \\
        & + q^{(|a_2| + |a_3|)|b_1| + 2|a_2| + 2|b_2|} \sum a_1 r_2(r_1(b_1)) \otimes (r_2(a_2) r_3(b_2))' \otimes (r_2(a_2) r_3(b_2))'' r_3(r_1(a_3))b_3.
    \end{align*}
    Take these two sums, and applies relation (\ref{eq:dumbbell1}) on each (with $x=a_1,y=r_2(r_1(b_1))$ for the first one, and $x=r_2(a_2),y=r_3(b_2)$ for the second one) to obtain the four sums
    \begin{align}
        &q^{(|a_2| + |a_3|)|b_1|} \sum a_1 r_2(r_1(b_1))' \otimes r_2(r_1(b_1))'' r_2(a_2) r_3(b_2) \otimes r_3(r_1(a_3))b_3, \label{eq:sum1} \\
        &q^{|a_2| + |a_3|)|b_1| + |b_1|} \sum a_1' r_4(r_2(r_1(b_1))) \otimes r_4(a_1'') r_2(a_2) r_3(b_2) \otimes r_3(r_1(a_3))b_3, \label{eq:sum2} \\
        &q^{(|a_2| + |a_3|)|b_1|} \sum a_1 r_2(r_1(b_1)) \otimes r_2(a_2) r_3(b_2)' \otimes r_3(b_2)'' r_3(r_1(a_3))b_3, \label{eq:sum3} \\
        &q^{|a_2| + |a_3|)|b_1| + |b_2|} \sum a_1 r_2(r_1(b_1)) \otimes r_2(a_2)' r_4(r_3(b_2)) \otimes r_4(r_2(a_2))'' r_3(r_1(a_3))b_3. \label{eq:sum4}
    \end{align}
    We manipulate these four sums separately. For the sum (\ref{eq:sum1}), apply relation (\ref{eq:switchcoprodb}) a first time (with $x=a_2, y=r_1(b_1)$) to obtain
    \[
        q^{(|a_2| + |a_3|)|b_1|} \sum a_1 r_2(r_1(b_1)') \otimes r_3(r_1(b_1)'') r_3(r_2(a_2)) r_4(b_2) \otimes r_4(r_1(a_3))b_3,
    \]
    and then a second time (with $x=a_3,y=b_1$) to obtain
    \[
        q^{(|a_2| + |a_3|)|b_1|} \sum a_1 r_3(r_1(b_1')) \otimes r_4(r_2(b_1'')) r_4(r_3(a_2)) r_5(b_2) \otimes r_5(r_2(r_1(a_3)))b_3.
    \]
    Then use relation (\ref{eq:detwist}) (with $x=r_3(a_2), y=r_2(b_1'')$),
    \[
        q^{(|a_2| + |a_3|)|b_1| + |a_2|} \sum a_1 r_3(r_1(b_1')) \otimes r_3(a_2) r_2(b_1'')  r_4(b_2) \otimes r_4(r_2(r_1(a_3)))b_3.
    \]
    This is equal to
    \[
        q^{(|a_2| + |a_3|)|b_1| + |a_2|} \sum a_1 r_4(r_1(b_1')) \otimes r_4(a_2) r_2(b_1'')  r_3(b_2) \otimes r_3(r_2(r_1(a_3)))b_3
    \]
    (we just interchanged $r_3$ and $r_4$). Finally, apply relation (\ref{eq:switchprodbc}) with $x=r_1(a_3),y=b_1'',z=b_2$ to obtain
    \[
        q^{(|a_2| + |a_3|)|b_1| + |a_2|} \sum a_1 r_3(r_1(b_1')) \otimes r_3(a_2) r_2(b_1''b_2) \otimes r_2(r_1(a_3)))b_3.
    \]
    For (\ref{eq:sum2}), interchange $r_4$ and $r_3$, then apply relation (\ref{eq:switchprodab}) with $x=a_1'',y=a_2,z=r_1(b_1)$ to obtain
    \[
        q^{(|a_2|+1+|a_3|)|b_1|} \sum a_1' r_2(r_1(b_1)) \otimes r_2(a_1''a_2) r_3(b_2) \otimes r_3(r_1(a_3))b_3.
    \]
    For (\ref{eq:sum3}), apply relation (\ref{eq:switchcoprodb}) with $x=a_3, y=b_2$, then (\ref{eq:detwist}) with $x=r_3(r_1(a_3)),y=b''_2$ to obtain
    \[
        q^{(|a_2| + |a_3|)|b_1| + 2|a_2| + |a_3| + 2|b_2|} \sum a_1 r_2(r_1(b_1)) \otimes r_2(a_2) r_3(b_2') \otimes r_3(r_1(a_3))b''_2b_3.
    \]
    Finally, for (\ref{eq:sum4}), apply (\ref{eq:switchcoproda}) with $x=a_2, y=r_1(b_1)$, then (\ref{eq:switchprodab}) with $x=r_2(a''_2), y=r_1(a_3), z=b_2$, and finally (\ref{eq:switchprodab}) with $x=a''_2, y=a_3,z=b_1$ to obtain
    \[
        q^{(|a_2| + |a_3|)|b_1|  + 2|a_2| +3|b_2|} \sum a_1 r_2(r_1(b_1)) \otimes r_2(a_2') r_3(b_2) \otimes r_3(r_1(a_2''a_3))b_3.
    \]
    To conclude, do the sum of the four sums, and compare with $\mu_B \delta$ (using $q^3 = 1$).
\end{proof}
\begin{lemma}
    For any $d \geq 0$, $(B_d,\delta)$ is a $3$-complex.
\end{lemma}
\begin{proof}
    For $B_0$ and $B_1$, this is trivial :
    \begin{align*}
        B_0 : \ & S^0_q \otimes S^0_q \otimes S^0_q,\\
        B_1 : \ & S^1_q \otimes S^0_q \otimes S^0_q \to S^0_q \otimes S^1_q \otimes S^0_q \to S^0_q \otimes S^0_q \otimes S^1_q.
    \end{align*}
    For $d > 1$, by applying $\mu_B$ several times, we have a surjective map $m : \underbrace{B_1 \otimes \cdots \otimes B_1}_{d} \to B_d$ such that $m \circ \delta_{B_1 \otimes \cdots \otimes B_1} = \delta \circ m$, but since $\delta^3_{B_1 \otimes \cdots \otimes B_1} = 0$ and $m$ is surjective, we conclude that $\delta^3 = 0$.
\end{proof}

\subsection{Cohomology of $B_d$}

In the rest of this article, we just write $ab$ for $\mu_B(a \otimes b)$. To show that $B_{3d}$ is a $3$-coresolution of $\twist{(S^{d})}$, we will use the exponential property to compute the homology of $B(n)$ from the homology of $B(1)$.
\begin{lemma}\label{lem : compatibility of exp with the differential}
    For any $n \geq 0$, the following diagram commutes.
    \begin{center}
        \begin{tikzcd}
            \bigoplus_{\alpha \in \Omega(d,n)} B_{\alpha_1}(1) \otimes \cdots \otimes B_{\alpha_n}(1) \ar[r, "\simeq"] \ar[d, "\delta_{B_{\alpha_1} \otimes \cdots \otimes B_{\alpha_n}}"]
            & B_d(n) \ar[d,"\delta"] \\
            \bigoplus_{\alpha \in \Omega(d,n)} B_{\alpha_1}(1) \otimes \cdots \otimes B_{\alpha_n}(1) \ar[r, "\simeq"]
            & B_d(n)
        \end{tikzcd}
    \end{center}
    Here the horizontal arrows are obtained by applying several times $\exp_B$.
\end{lemma}
Hence, as $\field$-vector space, the homology $H(B_d(n))$ of the Troesch complex is isomorphic to
\[
    \bigoplus_{\alpha \in \Omega(d,n)} H(B_{\alpha_1}(1)) \otimes \cdots \otimes H(B_{\alpha_n}(1))
\]
by Kunneth formula. In our case, by proposition \ref{prop:start}, the homology in degree $0$ of $B_{3d}$ is $\twist{(S^d)}$. So, using the above lemma and proposition \ref{prop:lkunneth}, it will be sufficient to show that $B_d(1)$ has homology concentrated in degree $0$ to show that $B_{3d}$ is a $3$-coresolution of $\twist{(S^d)}$.

To this end, we will need explicit description of the elements of $S^*_q(1)$, and the effect of the differential on the elements.
\begin{proposition}\label{prop:Sdescription}
    The algebra $S^*_q(1)$ is isomorphic to the free algebra on one variable $e$. Moreover, 
    \[
        \delta(e^{k_1} \otimes e^{k_2} \otimes e^{k_3}) = (k_1)_{q^2} (e^{k_1-1} \otimes e^{k_2+1} \otimes e^{k_3}) + (q^2)^{k_2}(k_2)_{q^2} (e^{k_1} \otimes e^{k_2-1} \otimes e^{k_3+1}).
    \]
\end{proposition}
This allows us to define morphisms of $3$-complexes, which will be useful to compute the homology.
\begin{notation}
    Let $C$ be a $3$-complex, and $t\geq 0$. We denote $C[t]$ the $3$-complex $C[t]^i = C^{i-t}$ with differential $\delta_{C[t]} : C[t]^i \to C[t]^{i+1}$ given by $\delta_C : C^{i-t} \to C^{i+1-t}$ (set $C^{i-t} = 0$ if $i-t < 0$).
\end{notation}
\begin{lemma}\label{lem:complexmorphism}
    Let $d,k_1,k_2,k_3 \geq 0$ and $k = k_1 + k_2 + k_3$. Suppose $3$ divide $k_1$ and $k_2$. Then 
     \[
         \psi : \left \{ \begin{array}{rcl}
             B_d(1)[k_2+2k_3] & \to & B_{d+k}(1) \\
             a & \mapsto & a(e^{k_1} \otimes e^{k_2} \otimes e^{k_3}) 
         \end{array} \right .
     \]
     is an injective morphism of $3$-complexes.
\end{lemma}
\begin{proof}
    Since the product commute with the differential,
    \[
        \delta(ab) = \delta(a)b + q^{2|a_2| + 4|a_3|} a \delta(b).
    \]
    Hence, to show that $\psi$ is a morphism of three complexes, it suffices to show $\delta(e^{k_1} \otimes e^{k_2} \otimes e^{k_3}) = 0$, but this follows from proposition \ref{prop:Sdescription}. Injectivity follows from the equality
    \[
        (e^{i_1} \otimes e^{i_2} \otimes e^{i_3})(e^{k_1} \otimes e^{k_2} \otimes e^{k_3}) = q^{something} (e^{i_1 + k_1} \otimes e^{i_2 + k_2} \otimes e^{i_3 + k_3})
    \]
    which can be proven using \cite[lemma 5.1]{hong2017quantum}.
\end{proof}
% \begin{lemma}\label{lem:complexmorphism}
%     Let $i \geq 0$. For $d \geq 0$, set $\alpha^d \in \Omega(d,n)$ be the compositions of $d$ in $n$ parts with $\alpha^d_i = d$ (and hence $\alpha^d_j = 0$ for $j \neq i$). Also, let $k_1,k_2,k_3 \geq 0$, $k=k_1+k_2+k_3$. If $3$ divide $k_1$ and $k_2$, then
%     \[
%         \psi : \left \{ \begin{array}{rcl}
%             B_{\alpha^d}[k_2+2k_3] & \to & B_{\alpha^{d+k}} \\
%             a & \mapsto & a(e_i^{k_1} \otimes e_i^{k_2} \otimes e_i^{k_3}) 
%         \end{array} \right .
%     \]
%     is an injective morphism of $3$-complexes.
% \end{lemma}
% \begin{proof}
%     Since the product commutes with the differentials (lemma \ref{lem:prodB}), we have
%     \[
%         \delta(ab) = \delta(a)b + q^{2|a_2| + 4|a_3|} a \delta(b).
%     \]
%     Hence, to show that $\psi$ is a morphism of three complexes, it suffices to show $\delta(e_i^{k_1} \otimes e_i^{k_2} \otimes e_i^{k_3}) = 0$, but this follows from \eqref{eq:deltavalue}. The injectivity follows from \eqref{eq:productvalue} and linear independence of the $e^\alpha$.
% \end{proof}

We can now prove our theorem.
\begin{theorem}\label{theo:cohomologyBd}
    The $3$-complex $B_d$ is acyclic when $3$ does not divide $d$, and if $d=3d'$, then $B_d$ is a $3$-coresolution of $\twist{(S^{d'})}$.
\end{theorem}
\begin{proof}
    As noted earlier, we will deduce the theorem from the fact that the $B_d(1)$ are acyclic when $3$ does not divide $d$, and are $3$-coresolution of $\field$ when $3$ divide $d$ (so in particular, by showing that their homology is concentrated in degree $0$). We show this by induction on $d$.

    For $d=0$, $B_0(1) = \field(1 \otimes 1 \otimes 1)$ concentrated in degree $0$, and hence is a $3$-coresolution of $\field$.

    For $d=1$,
    \[
        B_1(1) : \field (e \otimes 1 \otimes 1) \to \field (1 \otimes e \otimes 1) \to \field (1 \otimes 1 \otimes e)
    \]
    and both arrows are isomorphisms :
    \[
        \delta(e \otimes1 \otimes 1) = (1 \otimes e \otimes 1) \quad \mbox{and} \quad \delta(1 \otimes e \otimes 1) = q^2(1 \otimes 1 \otimes e).
    \]
    Hence, $B_1(1)$ is acyclic.

    For $d>1$, we consider three cases.
    \begin{itemize}
        \item $d \equiv 2$ mod $3$. Then, by lemma \ref{lem:complexmorphism}, we have an injective morphism of $3$-complex
        \[
            B_{d-1}(1)[2] \xrightarrow{\times (1 \otimes 1 \otimes e_i)} B_d.
        \]
        Its cokernel $Q$ is given by $Q^k = \field [e^{d-k} \otimes e^k \otimes 1]$ (where $[e^{d-k} \otimes e^k \otimes 1]$ denote the class of $e^{d-k} \otimes e^k \otimes 1$ in the quotient), with differential 
        \[
            \delta_Q([e^{d-k} \otimes e^k \otimes 1]) = (d-k)_{q^2} [e^{d-k-1} \otimes e^{k+1} \otimes 1].
        \]
        Thus, $\delta_Q : Q^k \to Q^{k+1}$ is equal to zero if $k \equiv 2$ mod $3$, and is an isomorphism otherwise. This shows that $Q$ is isomorphic as a $3$-complex to 
        \[
            \field \xrightarrow{=} \field \xrightarrow{=} \field \xrightarrow{0} \field \xrightarrow{=} \field \xrightarrow{=} \field \xrightarrow{0} \cdots \xrightarrow{0} \field \xrightarrow{=} \field \xrightarrow{=} \field \quad \mbox{(the complex stops in degree } d).
        \]
        Hence, $Q$ is acyclic (this is a direct sum of $3$-complexes of the form \ref{example:lacyclic}). Since moreover $B_{d-1}(1)[2]$ is acyclic and we have a short exact sequence
        \[
            0 \to B_{d-1}(1)[2] \to B_d(1) \to Q \to 0,
        \]
        we conclude that $B_d(1)$ is also acyclic.
        \item $d \equiv 0$ mod $3$. It is similar to the reasoning for $d \equiv 2$ mod $3$, but with $\delta_Q : Q^k \to Q^{k+1}$ equals to zero if $k \equiv 1$ mod $3$, and is an isomorphism otherwise. Thus, the cokernel is isomorphic to
        \[
            \field \xrightarrow{0}\field \xrightarrow{=} \field \xrightarrow{=} \field \xrightarrow{0} \field \xrightarrow{=} \field \xrightarrow{=} \field \xrightarrow{0} \cdots \xrightarrow{0} \field \xrightarrow{=} \field \xrightarrow{=} \field \quad \mbox{(stops in degree } d),
        \]
        and hence is a $3$-coresolution of $\field$. We then deduce from the short exact sequence that $B_d(1)$ is also a $3$-coresolution of $\field$.
        \item $d \equiv 1$ mod $3$. This is the most difficult case since $B_{d-1}(1)$ is not acyclic. We replace it by $B_{d-2}(1)$, which is acyclic. We have an injective morphism of $3$-complexes :
        \[
            \Psi_0 : B_{d-2}(1)[4] \xrightarrow{\times (1 \otimes 1 \otimes e^2)} B_{d}(1)
        \]
        whose cokernel $Q_0$ is given by 
        \[
            Q_0^k = \left \{ \begin{array}{cl}
                \field[e^{d-k} \otimes e^k \otimes 1] & \mbox{if } k=0,1 ; \\
                \field[e^{d-k} \otimes e^k \otimes 1] \oplus \field[e^{d-k+1} \otimes e^{k-2} \otimes e] & \mbox{otherwise.}
            \end{array} \right .
        \]
        Consider the cokernel $Q_1$ of $B_{d-3}(1)[2] \to B_{d-2}(1)$ as in the case $d \equiv 2$ mod $3$. It is acyclic, and given by $Q_1^k = \field [e^{d-k-2} \otimes e^k \otimes 1]$. Then the map
        \[
            B_{d-2}(1)[3] \xrightarrow{\times (1 \otimes e \otimes e)} B_{d}(1)
        \]
        is not a $3$-complex morphism, but pass to the quotient to a map
        \[
            \Psi_1 : Q_1[3] \to Q_0
        \]
        since for $a \in B_{\alpha^{d-3}}$,
        \[
            a(1 \otimes 1 \otimes e) (1 \otimes e \otimes e) = q a (1 \otimes e \otimes 1) (1 \otimes 1 \otimes e^2) \in \Image(\Psi_0). 
        \]
        Moreover, this quotient map is a morphism of $3$-complex, since for $a \in B_{d-2}(1)$,
        \[
            \delta(a(1 \otimes e \otimes e)) = \delta(a) (1 \otimes e \otimes e) + (q^2)^{|a|+1} \underbrace{a (1 \otimes 1 \otimes e^2)}_{\in \Image \Psi_0}
        \]
        and hence $\delta_{Q_0}(\Psi_1([a])) = \Psi_1([\delta(a)]) = \Psi_1(\delta_{Q_1}([a]))$. Moreover, $\Psi_1$ is injective :
        \[
            \Psi_1([e^{d-k-2} \otimes e^k \otimes 1]) = [e^{d-k-2} \otimes e^{k+1} \otimes e].
        \]
        Its cokernel $Q_2$ is thus given by
        \[
            Q_2^k = \left \{ \begin{array}{cl}
                \field[e^{d-2} \otimes e^2 \otimes 1] \oplus \field[e^{d-1} \otimes 1 \otimes e] & \mbox{if } k=2, \\
                \field[e^{d-k} \otimes e^k \otimes 1] & \mbox{otherwise.}
            \end{array} \right .
        \]
        Finally, we have a $3$-complex morphism
        \[
            \Psi_2 : B_{1}(1) \xrightarrow{\times (e^{d-1} \otimes 1 \otimes 1)} Q_2
        \]
        (since $3$ divide $d-1$, see lemma \ref{lem:complexmorphism}). Again, it is injective, and has cokernel $Q_3$ given by 
        \[
            Q_3^k = \left \{ \begin{array}{cl}
                \field[e^{d-k} \otimes e^k \otimes 1] & \mbox{if } 2 \leq k \leq d, \\
                0 & \mbox{otherwise.} 
            \end{array} \right .
        \]
        Moreover,
        \[
            \delta_{Q_3}([e^{d-k} \otimes e^k \otimes 1]) = (d-k) [e^{d-k-1} \otimes e^{k+1} \otimes 1]
        \]
        is zero if $k \equiv 1$ mod $3$, and is an isomorphism otherwise. Hence $Q_3$ is isomorphic to
        \[
            0 \to 0 \to \field \xrightarrow{=} \field \xrightarrow{=} \field \xrightarrow{0} \field \xrightarrow{=} \field \xrightarrow{=} \field \xrightarrow{0} \cdots \xrightarrow{0} \field \xrightarrow{=} \field \xrightarrow{=} \field
        \]
        and hence is acyclic. By short exact sequence arguments, we conclude that $B_d(1)$ is acyclic.
    \end{itemize}
    % This end the induction. Now, lemma \ref{lem:caseonepart} and proposition \ref{prop:lkunneth} show that $B_{\alpha}$ is $\ell$-acyclic if $3$ does not divide $\alpha$ (in the sense that it does not divide at least one part of $\alpha$) and is a $\ell$-coresolution of $\field$ if $3$ divide $\alpha$. In particular, the $3$-complex $B_d$ is acyclic if $3$ does not divide $d$, and if $d=3d'$, its cohomology is concentrated in degree $0$. Moreover, $\varphi_d(n) : \twist{(S^{d'})}(n) \to S^d_q(n)$ has image the kernel of $\delta = \Delta^{(d-1,1)} : S^d_q(n) \ (\cong S^d_q \otimes S^0_q \otimes S^0_q) \to (S^{d-1}_q \otimes S^1_q)(n) \ (\cong S^{d-1}_q \otimes S^1_q \otimes S^0_q)$ by proposition \ref{prop:start}. Thus, $B_{3d'}$ is a $3$-coresolution of $\twist{(S^{d'})}$.
\end{proof}
Now that it is done, we get all the results of section \ref{sec:comp} in the case $\ell=3$, and in particular the twisting spectral sequence.

\bibliography{biblio.bib}
\bibliographystyle{plain}

\end{document}